\documentclass[12pt]{amsart}
\usepackage{amssymb}
\usepackage{mathrsfs}
\usepackage{hyperref}
\usepackage{fullpage}

\newtheorem{theorem}{Theorem}[section]
\numberwithin{equation}{section}

\newtheorem{definition}[theorem]{Definition}
\newtheorem{remark}[theorem]{Remark}

\newcommand{\Aut}{\ensuremath{\mathrm{Aut}}}
\newcommand{\End}{\ensuremath{\mathrm{End}}}
\newcommand{\Fr}{\ensuremath{\mathrm{Fr}}}
\newcommand{\Frss}{\ensuremath{\mathrm{Fr\text{-}ss}}}
\newcommand{\Gal}{\ensuremath{\mathrm{Gal}}}
\newcommand{\gln}{\ensuremath{\operatorname{GL}}}
\newcommand{\id}{\ensuremath{\mathrm{id}}}
\newcommand{\intal}{\ensuremath{\mathrm{intal}}}
\newcommand{\loc}{\ensuremath{\mathrm{loc}}}
\newcommand{\modu}{\ensuremath{\mathrm{\;mod\;}}}
\newcommand{\ord}{\ensuremath{\mathrm{ord}}}
\newcommand{\Qbar}{\ensuremath{\overline{Q}}}
\newcommand{\res}{\ensuremath{\mathrm{res}}}
\newcommand{\Sp}{\ensuremath{\mathrm{Sp}}}
\newcommand{\Sym}{\ensuremath{\mathrm{Sym}}}
\newcommand{\tr}{\ensuremath{\mathrm{tr}}}
\newcommand{\ur}{\ensuremath{\mathrm{ur}}}
\newcommand{\WD}{\ensuremath{\mathrm{WD}}}

\newcommand{\ie}{{\it i.e.}}
\newcommand{\loccit}{{\it loc.\,cit}}
\newcommand{\hra}{\hookrightarrow}
\newcommand{\xra}{\xrightarrow}

\newcommand{\bbQ}{\ensuremath{\mathbb{Q}}}
\newcommand{\bbQbar}{\ensuremath{\overline\bbQ}}
\newcommand{\bbQp}{\ensuremath{{\mathbb{Q}_p}}}
\newcommand{\bbQpbar}{{\ensuremath{\overline{\mathbb{Q}}_p}}}
\newcommand{\bbQl}{\ensuremath{\mathbb{Q}_\ell}}

\newcommand{\bbS}{\ensuremath{\mathbb{S}}}
\newcommand{\bbZ}{\ensuremath{\mathbb{Z}}}
\newcommand{\bbZp}{\ensuremath{{\mathbb{Z}_p}}}

\newcommand{\calO}{\ensuremath{\mathcal{O}}}
\newcommand{\calQ}{\ensuremath{\mathcal{Q}}}
\newcommand{\calQbar}{\ensuremath{{\overline{\mathcal{ Q}} }  }}
\newcommand{\calT}{\ensuremath{\mathcal{T}}}

\newcommand{\fraka}{\ensuremath{\mathfrak{a}}}
\newcommand{\frakm}{\ensuremath{\mathfrak{m}}}
\newcommand{\frakp}{\ensuremath{\mathfrak{p}}}
\newcommand{\frakS}{\ensuremath{\mathfrak{S}}}

\newcommand{\scrO}{\ensuremath{\mathscr{O}}}

\begin{document}
\title{Variation of Weyl modules in $p$-adic families}
\author{Jyoti Prakash Saha}
\address{Max Planck Institute for Mathematics, Vivatsgasse 7, 53111 Bonn, Germany}
\email{saha@mpim-bonn.mpg.de}

\begin{abstract}
Given a Weil-Deligne representation with coefficients in a domain, we prove the rigidity of the structures of the Frobenius-semisimplifications of the Weyl modules associated to its pure specializations. Moreover, we show that the structures of the Frobenius-semisimplifications of the Weyl modules attached to a collection of pure representations are rigid if these pure representations lift to Weil-Deligne representations over domains containing a domain $\mathscr{O}$ and a pseudorepresentation over $\mathscr{O}$ parametrizes the traces of these lifts.
\end{abstract}
\maketitle

\section{Introduction}
The aim of this article is to study the variation of the Weyl modules in families of automorphic Galois representations. We show that the variation of the Weyl modules is related to the purity of $p$-adic automorphic Galois representations at the places outside $p$ and establish the rigidity of the structures of the Weyl modules at the arithmetic points of irreducible components of $p$-adic families. 

Let $p,\ell$ be two distinct primes, $K$ be a finite extension of $\bbQl$ and $\scrO$ be an integral domain containing $\bbQ$. Let $\bbS_\mu$ denote the Schur functor corresponding to a partition $\mu$ of a positive integer $d$. In theorem \ref{Thm: Weyl big}, we show that the structures of the Frobenius-semisimplifications of the Weyl modules associated to (\ie, $\bbS_\mu$ applied to) a collection of pure representations (\ie, representations whose associated monodromy filtrations and weight filtrations coincide up to some shift) of the Weil group $W_K$ over $\bbQpbar$ are ``rigid'' if these pure representations are specializations of a Weil-Deligne representation over $\scrO$. More generally, we prove that the structures of the Frobenius-semisimplifications of the Weyl modules attached to a collection of pure representations of $W_K$ over $\bbQpbar$ are ``rigid'' if these representations lift to Weil-Deligne representations over domains containing $\scrO$ and a pseudorepresentation $T:W_K \to \scrO$ parametrizes the traces of these lifts (see theorem \ref{Thm: Weyl pseudorepresentation}). 

The eigenvarieties provide examples of $p$-adic families of automorphic Galois representations. If $T$ denotes the pseudorepresentation associated to an eigenvariety $X$, then for any nonempty admissible open affinoid subset $U$ of $X$, the restriction of $T$ to $\calO(U)$ lifts to a Galois representation on a finite type module over some integral extension of the normalization of $\calO(U)$ (by \cite[Lemma 7.8.11]{BellaicheChenevierAsterisQUE}). However, this module is not known to be free. So theorem \ref{Thm: Weyl pseudorepresentation} cannot be applied to eigenvarieties to study the Weyl modules associated to all arithmetic points. To circumvent this problem, we prove theorem \ref{Thm: Weyl sum} which applies to the Weyl modules associated to the arithmetic points whose associated Galois representations are absolutely irreducible and give pure representations when restricted to decomposition groups at the places outside $p$. Theorem \ref{Thm: Weyl big}, \ref{Thm: Weyl pseudorepresentation}, \ref{Thm: Weyl sum} are generalizations of \cite[Theorem 4.1, 5.4, 5.6]{BigPuritySub} and hence apply to $p$-adic families of automorphic representations (in the same way as \cite[Theorem 4.1, 5.4, 5.6]{BigPuritySub} are applicable, see section 6 of \loccit.) since automorphic Galois representations are known to be pure in many cases. We illustrate such an application in theorem \ref{Thm: Application: Hida family}, which explains the variation of the symmetric powers of the Galois representations attached to the arithmetic points of the Hida family of ordinary cusp forms.

\section{Preliminaries}
For every field $F$, fix an algebraic closure $\overline F$ of it and let $G_F$ denote the absolute Galois group $\Gal(\overline F/F)$. Given a domain $A$, let $Q(A)$ denote its fraction field, $\overline Q(A)$ denote the field $\overline{Q(A)}$, $A^\intal$ denote the integral closure of $A$ in $\Qbar (A)$. For every map $f:A\to B$ between domains, fix an extension $f^\intal :A^\intal \to B^\intal$ of $f$.

Let $\mu$ be a partition of a positive integer $d$. Let $\frakS_d$ denote the group of all permutations of the set $\{1, \cdots, d\}$ and $c_\mu$ denote the Young symmetrizer in $\bbZ[ \frakS_d]$ attached to $\mu$ (see \cite[\S 4.1]{FultonHarrisRepresentationTheory} for more details). Then there exists a positive integer $n_\mu$ such that $c_\mu^2 = n_\mu c_\mu$ in $\bbZ[\frakS_d]$ (see \cite[Lemma 4.26]{FultonHarrisRepresentationTheory}). For any module $M$ over a commutative ring $R$, the $R$-module $M^{\otimes d}$ carries a left $R$-linear action of $\Aut_R (M)$ and a right action of $\frakS_d$ given by 
$(m_1\otimes \cdots \otimes m_d)\cdot \sigma = m_{\sigma (1)} \otimes \cdots \otimes m_{\sigma(d)}$. 
This right action commutes with the left action of $\Aut_R(M)$. If $n_\mu$ is invertible in $R$, then denote the image of $M^{\otimes d}$ under $c_\mu$ by $\bbS_\mu M$, which is again an $R$-module and carries an $R$-linear left action of $\Aut_R(M)$. We call the functor $M \rightsquigarrow \bbS_\mu M$ the {\it Schur functor} or {\it Weyl module} corresponding to $\mu$. If $V$ is a vector space over a field $F$ of characteristic zero, then the dimension of $\bbS_\mu V$ is equal to $\dim_\bbQ \bbS_\mu (\bbQ^{\dim_FV})$. We denote this integer by $d(\mu, \dim_FV)$. Moreover, if $M$ carries an $R$-linear action of a group $G$ and $n_\mu$ is invertible in $R$, then $\bbS_\mu M$ also inherits an action of $G$, \ie, given a representation $\rho:G\to \Aut_R(M)$, the module $\bbS_\mu M$ can be considered as an $R$-linear representation of $G$ via the composite map $G\xra{\rho} \Aut_R(M) \to \Aut_R( \bbS_\mu M)$. We denote by $\bbS_\mu \rho$ the $R$-module $\bbS_\mu M$ together with the $R$-linear action of $G$ given by this composite map. 

\begin{remark}
If $\mu$ is equal to the partition $d=d$ (resp. $d=1+\cdots+1$), then for any $\bbZ[1/n_\mu]$-module $M$, the module $\bbS_\mu M$ is equal to $\Sym^d M$ (resp. $\wedge^d M$). For example, if $\rho:G \to \gln_n(\bbQpbar)$ is a representation, then $\Sym^d \rho$ denotes the $\bbQpbar$-vector space $\Sym^d (\bbQbar_p^n)$ together with the $\bbQpbar$-linear action of $G$ given by the composite map 
$$G\xra{\rho} \gln_n(\bbQpbar) =\Aut_{\bbQpbar}(\bbQbar^n_p)\to \Aut_{\bbQpbar} (\Sym^d (\bbQbar_p^n)).$$
\end{remark}

Let $k$ denote the residue field of the ring of integers $\calO_K$ of $K$ and $\phi$ denote an element of $G_K$ which maps to the geometric Frobenius element $\Fr_k\in G_k$. The Weil group $W_K$ is defined as the subgroup of $G_K$ consisting of elements which map to integral powers of $\Fr_k$ in $G_k$. Define $v_K:W_K\to \bbZ$ by $\sigma|_{K^\ur}=\Fr_k^{v_K(\sigma)}$ for all $\sigma\in W_K$. In the following, $A$ denotes a commutative integral domain of characteristic zero. 

\begin{definition}[{\cite[8.4.1]{DeligneConstantesDesEquationsFunctional}}]
A \textnormal{Weil-Deligne representation} of $W_K$ on a free module $M$ of finite type over $A$ is a pair $(r,N)$ consisting of a representation $r:W_K\to \Aut_A (M)$ with open kernel and a nilpotent element $N$ of $\End_A(M)$ such that 
$$r(\sigma)N r(\sigma)^{-1}=(\# k)^{-v_K(\sigma)}N$$
for any $\sigma\in W_K$. A representation $\rho$ of $W_K$ on $M$ is said to be \textnormal{irreducible Frobenius-semisimple} if $\rho$ is has open kernel, $M\otimes \Qbar(A)$ is irreducible and the $\phi$-action on $M\otimes \Qbar(A)$ is semisimple. 
\end{definition}

Given a Weil-Deligne representation $(r, N)$ of $W_K$ on a vector space $V$ with coefficients in an algebraically closed field of characteristic zero, we denote its Frobenius-semisimplification by $(r, N)^\Frss$ (see \cite[p. 570]{DeligneConstantesDesEquationsFunctional}). 

\begin{definition}
Given Weil-Deligne representations $(r_1, N_1)$ on an $A$-module $M_1$ and $(r_2, N_2)$ on an $A$-module $M_2$, their \textnormal{tensor product} is defined as the Weil-Deligne representation $(r_1\otimes r_2, \id_{M_1}\otimes N_2+N_1\otimes \id_{M_2})$ on $M_1\otimes_A M_2$. 
\end{definition}

\begin{definition}
Let $(r, N)$ be a Weil-Deligne representation of $W_K$ on a free module $M$ over a domain $R$ containing $\bbQ$. If $\bbS_\mu M$ is free over $R$, then the Weil-Deligne representation $\bbS_\mu (r, N)$ is defined as $(r, N)^{\otimes d} |_{\bbS_\mu M}$. 
\end{definition}

\section{Control theorems for Weyl modules}
The following results are the analogues of \cite[Theorem 4.1, 5.4, 5.6]{BigPuritySub} in the context of Weyl modules. 
\begin{theorem}
\label{Thm: Weyl big}
Let $(r, N): W_K \to \gln_n(\scrO)$ be a Weil-Deligne representation. Let $m, t_1, \cdots ,t_m$ be positive integers, $r_1, \cdots, r_m$ be irreducible Frobenius-semisimple representations of $W_K$ over $\scrO^\intal$ such that 
\begin{equation}
\label{Eqn: Isom Weyl big 1}
\left(\bbS_\mu \left((r, N) \otimes_\scrO \Qbar (\scrO) \right) \right)^\Frss \simeq 
\bigoplus_{i=1}^m \Sp_{t_i} (r_i).
\end{equation}
If $f \circ (r, N)$ is pure for some map $f: \scrO \to \bbQpbar$, then the Weil-Deligne representations $(\bbS_\mu (f \circ (r, N))) ^\Frss$ and $\oplus_{i=1}^m \Sp_{t_i} (f^\dag \circ r_i)$ are isomorphic for any lift $f^\dag: \scrO^\intal \to \bbQpbar$ of $f$. Moreover, there exist $m, t_i, r_i$ with the above-mentioned properties such that equation \eqref{Eqn: Isom Weyl big 1} holds.
\end{theorem}

\begin{proof}
Let $\calT$ denote the $\scrO$-module $\scrO^n$ and $\frakp$ denote the kernel of $f$. Note that $\calT_\frakp^{\otimes d}$ decomposes as the direct sum $\bbS_\mu (\calT_\frakp) \oplus \calT_\frakp^{\otimes d} (n_\mu-c_\mu)$ (as $\End_\scrO(\calT)$-modules). Moreover these summands are free over $\scrO_\frakp$ by \cite[Proposition 3.G]{MatsumuraCommutativeAlgebra}. So $\bbS_\mu (\calT _\frakp) \otimes_{\scrO_\frakp, f} \bbQpbar$ is isomorphic to its image in $(\calT^{\otimes d})_\frakp \otimes_{\scrO_\frakp, f} \bbQpbar$, \ie, to $\bbS_\mu (f \circ (r, N))$ as Weil-Deligne representations. Since $(f \circ (r, N))^{\otimes d}$ is pure (by \cite[Proposition 1.6.9]{DeligneWeil2}), the representation $\bbS_\mu (f \circ (r, N))$ is also pure. So the result follows from \cite[Theorem 4.1]{BigPuritySub}. 
\end{proof}

\begin{theorem}
\label{Thm: Weyl pseudorepresentation} 
Let $T: W_K \to \scrO$ be a pseudorepresentation. Let $(r, N):W_K \to \gln_n(\calO)$ be a Weil-Deligne representation over a domain $\calO$ such that $\res\circ T = \tr r$ for an injective map $\res:\scrO \hra\calO$. Suppose $m, t_1, \cdots, t_m$ are positive integers, $r_1, \cdots, r_m$ are irreducible Frobenius-semisimple representations of $W_K$ over $\scrO^\intal$ such that 
\begin{equation}
\label{Eqn: Isom 123}
\left( \bbS_\mu\left((r, N)\otimes_{\calO}\overline Q(\calO) \right) \right)^\Frss
\simeq 
\bigoplus_{i=1}^m \Sp_{t_i} (\res^\intal \circ r_i).
\end{equation}
Suppose $f \circ (r, N)$ is pure for some map $f: \calO \to \bbQpbar$. Then for any Weil-Deligne representation $(r', N'): W_K\to \gln_n(\calO')$ over a domain $\calO'$ such that $\tr r'$ is equal to $\res' \circ T$ for some injective map $\res':\scrO \hra \calO'$ and $f'\circ (r', N')$ is pure for some map $f': \calO' \to \bbQpbar$, 
the Weil-Deligne representations $\left( \bbS_\mu\left((r', N')\otimes_{\calO'}\overline Q(\calO') \right) \right)^\Frss, \left( \bbS_\mu\left(f'\circ (r', N')\right) \right)^\Frss$ are isomorphic to $\oplus_{i=1}^m \Sp_{t_i} (\res'^\dag \circ r_i), \oplus_{i=1}^m \Sp_{t_i} (f'^\dag \circ \res'^\dag \circ r_i)$ respectively for any lift $\res'^\dag:\scrO^\intal \to \calO'^\intal$ of $\res'$ and $f'^\dag:\calO'^\intal \to \bbQpbar$ of $f'$.
Furthermore, there exist $m, t_i, r_i$ with the above-mentioned properties such that equation \eqref{Eqn: Isom 123} holds. 
\end{theorem}

\begin{proof}
Let $\rho$ be a representation of $W_K$ over $\Qbar(\scrO)$ such that its trace is equal to $T$. Let $\frakp, \frakp'$ denote the kernel of $f, f'$ respectively. Note that $\bbS_\mu ((r, N)\otimes_\calO \calO_\frakp)$ is a Weil-Deligne representation on $\calO_\frakp^{d(\mu, n)}$ and $\bbS_\mu (f \circ (r, N))= f \circ (\bbS_\mu ((r, N)\otimes_\calO \calO_\frakp))$ is pure by \cite[Proposition 1.6.9]{DeligneWeil2}. Since $\res ^\intal \circ \rho$ and $r$ have equal traces, the representations $\res^\intal \circ \bbS_\mu \rho, \bbS_\mu r$ also have equal traces (by \cite[Theorem 6.3.(3)]{FultonHarrisRepresentationTheory}, for instance) and hence $\res^\intal \circ \tr \bbS_\mu \rho$ is equal to the trace of $\bbS_\mu r$. Similarly, $\bbS_\mu ((r', N')\otimes_{\calO'} \calO'_{\frakp'})$ is a Weil-Deligne representation on ${\calO'} _{\frakp'}^{d(\mu, n)}$, the representation $\bbS_\mu (f'\circ (r', N')) = f' \circ (\bbS_\mu ((r', N')\otimes_{\calO'} \calO'_{\frakp'}))$ is pure and $\res'^\intal \circ \tr \bbS_\mu \rho$ is equal to the trace of $\bbS_\mu r'$. Also note that $\tr \bbS_\mu \rho$ is an $\scrO^\intal$-valued pseudorepresentation of $W_K$. Thus the result follows from \cite[Theorem 5.4]{BigPuritySub}. 
\end{proof}

Suppose $\scrO$ is a $\bbZp$-algebra and let $w\nmid p$ denote a finite place of a number field $F$. Let $T, T_1, \cdots, T_n$ be $\scrO$-valued pseudorepresentations of $G_F$ such that $T= T_1 + \cdots + T_n$. Using \cite[Theorem 1]{TaylorGaloisReprAssociatedToSiegelModForms}, we choose semisimple representations $\sigma_1, \cdots, \sigma_n$ of $G_F$ over $\Qbar(\scrO)$ such that $\tr \sigma_i = T_i$ for all $1\leq i \leq n$. The {\it irreducibility and purity locus} of $T_1, \cdots, T_n$ is defined to be the collection of tuples $(\calO,  \frakm, \kappa,  \loc,\rho_1$, $\cdots, \rho_n)$ where $\calO$ is a Henselian Hausdorff domain and it is a $\bbZp$-algebra, $\frakm$ denotes its maximal ideal, $\kappa$ denotes its residue field which is an algebraic extension of $\bbQp$, $\loc:\scrO \hra\calO$ is an injective $\bbZp$-algebra map, $\rho_1, \cdots, \rho_n$ are irreducible $G_F$-representations over $\overline \kappa$ such that their traces are equal to $\loc \circ T_1\modu \frakm, \cdots, \loc \circ T_n\modu \frakm$ respectively and their restrictions to the decomposition group of $F$ at $w$ are pure. Given such a tuple $(\calO, \frakm, \kappa, \loc, \rho_1, \cdots, \rho_n)$, we use \cite[Th\'eor\`eme 1]{NyssenPseudoRepresentations} to choose semisimple $G_F$-representations $\widetilde \rho_1$, $\cdots$, $\widetilde \rho_n$ over $\calO$ such that $\tr \widetilde \rho_i= \loc \circ T_i$ for all $1\leq i\leq n$. 

\begin{theorem}
\label{Thm: Weyl sum}
Assume that the restrictions of $\sigma_1, \cdots, \sigma_n$ to the Weil group $W_w$ of $F$ at $w$ are potentially unipotent. Then there exist positive integers $m, t_1,  \cdots , t_m$, irreducible Frobenius-semisimple representations $r_1, \cdots, r_m$ of $W_w$ over $\scrO^\intal$ such that 
\begin{equation}
\label{Eqn: purity sum sigma}
\WD \left(\bigoplus_{i=1}^n \bbS_\mu \sigma_i |_{W_w} \right)^\Frss
\simeq 
\bigoplus_{i=1}^m \Sp_{t_i} (r_i)_{/\Qbar(\scrO)}
\end{equation}
and there are isomorphisms of Weil-Deligne representations 
\begin{align}
\WD \left(\bigoplus_{i=1}^n \bbS_\mu \rho_i |_{W_w} \right)^\Frss
& \simeq 
\bigoplus_{i=1}^m \Sp_{t_i} (\pi_{\frakm}^\intal\circ \loc^\intal \circ r_i)_{/\overline \kappa}, \label{Eqn: purity sum rho prime} \\
\WD \left(\bigoplus_{i=1}^n \bbS_\mu \widetilde \rho_i |_{W_w} \otimes \Qbar(\calO)\right)^\Frss
&\simeq 
\bigoplus_{i=1}^m \Sp_{t_i} ( \loc^\intal \circ r_i) _{/\Qbar(\calO)}\label{Eqn: purity sum tilde rho}
\end{align}
for any element $(\calO, \frakm, \kappa, \loc, \rho_1, \cdots, \rho_n)$ in the irreducibility and purity locus of $T_1$, $\cdots$, $T_n$ where $\pi_\frakm$ denotes the $\mathrm{mod}\, \frakm$ reduction map $\calO \to \calO/\frakm$. 
\end{theorem}

\begin{proof}
If the irreducibility and purity locus of $T_1, \cdots, T_n$ is empty, then it remains to prove equation \eqref{Eqn: purity sum sigma}, which follows from \cite[Proposition 3.1.3 (i)]{DeligneFormesModulairesGL2}. So we assume that this locus is nonempty. Note that the representation $\widetilde \rho_i$ is irreducible since the mod $\frakm$ reduction of its trace is equal to the trace of the irreducible representation $\rho_i$. Since the $G_F$-representations $\sigma_i\otimes \Qbar(\calO), \widetilde \rho_i$ have equal traces, they are isomorphic. Since $\sigma_i|_{W_w}$ is potentially unipotent, the representation $\widetilde \rho_i|_{W_w}$ is also potentially unipotent. Consequently, its Weil-Deligne parametrization $\WD(\widetilde \rho_i|_{W_w})$ is defined and has coefficients in $\calO$. Moreover, the trace of $\WD(\widetilde \rho_i|_{W_w})$ is equal to $\loc \circ T_i|_{W_w} = \loc \circ \tr \sigma_i$ and the mod $\frakm$ reduction of $\WD(\widetilde \rho_i|_{W_w})$ is isomorphic to the pure representation $\WD(\rho_i|_{W_w})$. So the isomorphisms in equation \eqref{Eqn: purity sum rho prime} and \eqref{Eqn: purity sum tilde rho} follow from theorem \ref{Thm: Weyl pseudorepresentation}. Then equation \eqref{Eqn: purity sum tilde rho} gives equation \eqref{Eqn: purity sum sigma} since $\sigma_i\otimes \Qbar(\calO)$ is isomorphic to $\widetilde \rho_i$.
\end{proof}

\section{Weyl modules in families}
In this section, we prove a control theorem for the symmetric powers of the Galois representations attached to arithmetic points of Hida family ordinary cusp forms. To prove this result, we use purity of Galois representations attached to cusp forms.

Given a normalized eigen cusp form $f=\sum_{n=1}^\infty a_n q^n$ of weight at least two, from the works of Eichler \cite{Eichler54}, Shimura \cite{ShimuraCorrespondances}, Deligne \cite{DeligneModFormAndlAdicRepr} and Ribet \cite[Theorem 2.3]{RibetGalReprAttachedToEigenformsWithNebentypus}, it follows that there exists a continuous Galois representation $\rho_f:G_\bbQ \to \gln_2(\bbQpbar)$ (unique up to equivalence) such that the trace of $\rho_f(\Fr_\ell)$ is equal to $a_\ell$ for any prime $\ell$ not dividing the product of $p$ and the level of $f$. 

Let $p$ be an odd prime, $N$ be a positive integer such that $Np\geq 4$ and $p\nmid N$. Choose a minimal prime ideal $\fraka$ of the universal $p$-ordinary Hecke algebra $h^\ord$ of tame level $N$ (denoted $h^\ord(N; \bbZp)$ in \cite{HidaICM86}). The ring $h^\ord$ is an algebra over $\bbZp[[X]]$. An {\it arithmetic specialization} of $h^\ord$ is a $\bbZp$-algebra map $\lambda:h^\ord \to \bbQpbar$ such that $\lambda((1+X)^{p^r} - (1+p)^{(k-2)p^r})=0$ for some integers $k\geq 2$ and $r\geq 0$. Denote the quotient ring $h^\ord/\fraka$ by $R(\fraka)$, its fraction field by $\calQ(\fraka)$ and fix an algebraic closure $\calQbar(\fraka)$ of $\calQ(\fraka)$. By \cite[Theorem 3.1]{HidaICM86}, there exists a unique (up to equivalence) continuous (in the sense of \cite[\S 3]{HidaICM86}) absolutely irreducible two-dimensional Galois representation $\rho_\fraka$ of $G_\bbQ$ over $\calQbar(\fraka)$ with traces in $R(\fraka)$ and satisfying $\tr(\rho_\fraka(\Fr_\ell)) = T_\ell\modu \fraka$ for all prime $\ell$ not dividing $Np$ where $T_\ell\in h^\ord$ denotes the Hecke operator associated to the prime $\ell$. By the isomorphism of \cite[Theorem 2.2]{HidaICM86}, there is a one-to-one correspondence between the arithmetic specializations of $h^\ord$ and the $p$-ordinary $p$-stabilized (in the sense of \cite[p.~538]{WilesOrdinaryLambdaAdic}) normalized eigen cusp forms of weight at least $2$ and tame level a divisor of $N$. Moreover, the trace of $\rho_{f_\lambda}$ is equal to $\lambda \circ \tr \rho_\fraka$ for any arithmetic specialization $\lambda$ of $h^\ord$ with $\lambda(\fraka)=0$ where $f_\lambda$ denotes the ordinary form associated to $\lambda$. 

\begin{theorem}
\label{Thm: Application: Hida family}
Let $\ell\neq p$ be a prime and $W_\ell$ denote the Weil group of $\bbQl$. Then there exist positive integers $m, t_1, \cdots , t_m$ and irreducible Frobenius-semisimple representations $r_1, \cdots, r_m$ of $W_\ell$ with coefficients in $R(\fraka)^\intal[1/p]$ such that $\WD(\Sym^d \rho_\fraka |_{W_\ell} )^\Frss$ is isomorphic to $\oplus_{i=1}^m \Sp_{t_i} (r_i)$ and for any arithmetic specialization $\lambda$ of $R(\fraka)$, the representations $\WD(\Sym^d \rho_\lambda|_{W_\ell})^\Frss$ and $\oplus_{i=1}^m \Sp_{t_i} (\lambda^\intal \circ r_i)$ are isomorphic.
\end{theorem}

\begin{proof}
Since the representation $\rho_\fraka$ is continuous and $\ell\neq p$, by Grothendieck's monodromy theorem (see \cite[p.\,515--516]{SerreTate}), the action of the inertia subgroup $I_\ell$ on $\rho_\fraka$ is potentially unipotent. So the Weil-Deligne parametrization $\WD(\rho_\fraka|_{W_\ell})$ of $\rho_\fraka|_{W_\ell}$ is defined and has coefficients in $R(\fraka)[1/p]$. Its $\lambda$-specialization is isomorphic to $\WD(\rho_{f_\lambda})$, which is pure by \cite{Carayol86RepresentationAssocieHiblModForm}. Then the result follows from theorem \ref{Thm: Weyl big}.
\end{proof}

\subsection*{Acknowledgements}
It is my pleasure to thank Olivier Fouquet for his support and advice during the preparation of this article. I acknowledge the financial support from the ANR Projet Blanc ANR-10-BLAN 0114, the Mathematisches Forschungsinstitut Oberwolfach and the Max Planck Institute for Mathematics. I thank Santosh Nadimpalli for many useful discussions. 

\def\cprime{$'$} \def\polhk#1{\setbox0=\hbox{#1}{\ooalign{\hidewidth
  \lower1.5ex\hbox{`}\hidewidth\crcr\unhbox0}}}

\end{document}